\documentclass[10pt,a4paper]{article}
\usepackage{fullpage}
\usepackage{amsfonts,amsmath,amssymb}
\usepackage{graphicx}\usepackage{sectsty}
\newtheorem{theorem}{Theorem}[section]
\newtheorem{lemma}[theorem]{Lemma}
\newtheorem{definition}[theorem]{Definition}
\newtheorem{proposition}[theorem]{Proposition}

\newenvironment{proof}[1][Proof]{\begin{trivlist}
\item[\hskip \labelsep {\bfseries #1}]}{\end{trivlist}}
\numberwithin{equation}{section}

\sectionfont{\large}

\newcommand{\curl}[1]{\nabla\times #1}

\begin{document}
\title{A Global Uniqueness on Spherically Stratified Dielectric Medium in Time-Harmonic Maxwell Equation\\with Interior Transmission Eigenvalues}
\author{Lung-Hui Chen$^1$}\maketitle\footnotetext[1]{Department of
Mathematics, National Chung Cheng University, 168 University Rd.
Min-Hsiung, Chia-Yi County 621, Taiwan. Email:
mr.lunghuichen@gmail.com;\,lhchen@math.ccu.edu.tw. Fax:
886-5-2720497.}
\begin{abstract}
A set of regularly distributed transmission eigenvalues generates
a density function. We use such a density function to inversely
determine the form of the indicator function. Using the entire
function theory, we reduce an uniqueness problem with interior
transmission eigenvalues induced by time-harmonic Maxwell equation
to an uniqueness problem in entire function theory. In such an
inverse problem, the definite integral of the square root of
refraction index is the main parameter.
\end{abstract}
\section{Introduction and Preliminaries}
In this paper, we consider the time-harmonic Maxwell equation with
non-absorbing refraction index in the following setting:
\begin{equation} \label{1.1}
n(x)=n(r)>0,\, r=|x|,\,\mbox{ when }r\in[0,1];\, \Im n=0;\,
n(r)=1,\mbox{ when }r\geq 1;\,n\in\mathcal{C}^2[0,\infty);
\end{equation}
such that
\begin{eqnarray}\label{1.2}
\left\{%
\begin{array}{ll}
    \curl{E_1}-ikH_1=0, & \curl{H_1}+ikn(r)E_1=0,\mbox{ in }B;\\
    \curl{E_0}-ikH_0=0, & \curl{H_0}+ikE_0=0,\mbox{ in }B; \\
\end{array}%
\right.
\end{eqnarray}
with boundary condition
\begin{equation}\label{1.3}
\nu\times(E_1-E_0),\,\nu\times(H_1-H_0)=0,\mbox{ on }\partial B,
\end{equation}
where $E_0$, $H_0$ is an electromagnetic Herglotz pair, $B$ is an
open ball of radius $1$ in $\mathbb{R}^3$ with exterior unit
normal vector $\nu$.   We will look for a non-trivial solution to
this homogeneous electromagnetic interior transmission
problem~(\ref{1.2}) and~(\ref{1.3}). For each $k\in\mathbb{C}$
such that~(\ref{1.2}) and~(\ref{1.3}) has a set of non-trivial
solution is called an interior transmission eigenvalue. We reduce
such an electromagnetic interior transmission problem to the
acoustic interior transmission problem:
\begin{eqnarray}\label{1.4}
\left\{%
\begin{array}{ll}
    \Delta w+k^2n(r)w=0,  & \hbox{ in }B; \\
    \Delta v+k^2v=0 & \hbox{ in }B; \\
    w=v, \frac{\partial w}{\partial \nu}=\frac{\partial v}{\partial \nu}& \hbox{ on }\partial B,
\end{array}%
\right.
\end{eqnarray}
where $w,v\in\mathcal{C}^3(B)$. To see that, we consider the
following quantity
\begin{eqnarray}\label{1.5}
\left\{%
\begin{array}{ll}
    E_1(x):=\curl\{x w(x)\};\\
    H_1(x):=\frac{1}{ik}\curl\{E_1(x)\}; \\
    E_0(x):=\curl\{x v(x)\}; \\
    H_0(x):=\frac{1}{ik}\curl\{E_0(x)\},\\
\end{array}%
\right.
\end{eqnarray}
from which one can obtain a set of solution to the electromagnetic
interior transmission problem~(\ref{1.2}) and~(\ref{1.3}). We
refer the induction to the Colton and Kress \cite{Colton}.
\par
We need to consider the solutions $w,v$ to~(\ref{1.4}) that are
not spherically symmetric. Therefore, we look for non-trivial
solutions $w,v$ in the following form:
\begin{eqnarray}
&&v(r,\theta)=a_lj_l(kr)P_l(\cos\theta);\label{1.99}\\\label{1.10}
&&w(r,\theta)=b_l\frac{y_l(r)}{r}P_l(\cos\theta),
\end{eqnarray}
where $P_l$ is Legendre's polynomial, $j_l$ is the spherical
Bessel function of degree $l$, $a_l$ and $b_l$ depend on $k$ and
the function $y_l$ is a solution of
\begin{eqnarray}\label{1.8}
&&y_l''+(k^2n(r)-\frac{l(l+1)}{r^2})y_l=0,\,0<r<1,  \\
&&\lim_{r\rightarrow0}\{\frac{y_l(r)}{r}-j_l(k r)\}=0, \label{1.9}
\end{eqnarray}
where $y_l$ is continuous for $r\geq0$. Moreover, as demonstrated
in \cite{Colton}, we consider the non-spherically symmetric $w,v$.
In such a magnetic problem, we are asked to consider $l\geq1$.
Furthermore, we see from~(\ref{1.9}) that
\begin{equation} \label{110}
y_l(0)=0;\,y_l'(0)=0.
\end{equation}
We will show there exist a set of $k\in\mathbb{C}$ with its
maximal density and constants $a_l=a_l(k)$, $b_l=b_l(k)$, such
that~(\ref{1.99}) and~(\ref{1.10}) is a set of non-trivial
solution to the interior transmission problem~(\ref{1.4}).
Considering~(\ref{1.5}), we see that, for any such value of $k$,
the set of the electric far field patterns is not complete in
certain functional space. See the discussion in \cite{Colton}.

\par
The the interior transmission problem~(\ref{1.4}) and~(\ref{1.5})
admits  a set of non-trivial solution $v,w$ if there exists a set
of non-trivial solutions $a_l,b_l$ to the following homogenous
system
\begin{eqnarray} \label{1.11}
&&b_ly_l(1)-a_lj_l(k)=0;\\   \label{22}
&&b_l\frac{d}{dr}(\frac{y_l(r)}{r})|_{r=1}-a_lkj_l'(k)=0.
\end{eqnarray}
Such a system admits a set of non-trivial solutions
$a_l=a_l(k),b_l=b_l(k)$ if and only if the determinant
\begin{equation}\label{1.12}
d_l(k):=\det\left(%
\begin{array}{cc}
  y_l(1) & -j_l(k) \\
  \frac{d}{dr}(\frac{y_l(r)}{r})|_{r=1} & -kj_l'(k) \\
\end{array}%
\right)=0.
\end{equation}

\par
In this paper, we apply the following setting.
\begin{eqnarray}
&&l=1;\\
&&b_1(k)=b(k),\,a_1(k)=a(k);\\
&&d_1(k):=D(k),\,k\in\mathbb{C};\label{113}\\
&&y_1(r;k):=y(r;k),\,k\in\mathbb{C}.
\end{eqnarray}
Hence,
\begin{equation}
D(k)=(-k)y(1;k)j_1'(k)+y'(1;k)j_1(k)-y(1;k)j_1(k),
\end{equation}
where
\begin{equation}
j_1(t)=\frac{\sin t}{t^2}-\frac{\cos t}{t}.
\end{equation}
\par We consider in the spherical coordinate $(r,\theta,\vartheta)$ that
\begin{eqnarray}
&&\Phi(r;k):=b(k)y(r;k);\label{120}\\
&&\Phi_0(r;k):=a(k)rj_1(kr)\label{121}.
\end{eqnarray}
Hence, using~(\ref{110}),
\begin{eqnarray}
&&\Phi(0;k)=0;\label{122}\\
&&\partial_r\Phi(0;k)=b(k)\partial_ry(r;k)|_{\{r=0\}}=0;\label{123}
\end{eqnarray}
using~(\ref{1.11}) and~(\ref{22}),
\begin{eqnarray}
&&\Phi(1;k)=\Phi_0(1;k);\\
&&\partial_r\Phi(1;k)=b(k)\partial_ry(r;k)|_{\{r=1\}}=\partial_r\Phi_0(1;k).
\end{eqnarray}
Therefore, it is obviously that
\begin{eqnarray}\label{1.13}
\left\{%
\begin{array}{ll}
    \partial_{rr}\Phi+(k^2n(r)-\frac{2}{r^2})\Phi=0, &0<r<1; \\
     \partial_{rr}\Phi_0+(k^2-\frac{2}{r^2})\Phi_0=0, &0<r<1;\\
    \Phi(0)=\Phi_0(0)=0,\,\Phi(1)=\Phi_0(1),\,\partial_r\Phi(1)=\partial_r\Phi_0(1). \\
\end{array}%
\right.
\end{eqnarray}
We also observe that initial value
problem~(\ref{122}),~(\ref{123}) and the first equation
in~(\ref{1.13}) imply
\begin{equation}
\Phi(r;k)=y(r;k).
\end{equation}
This is due to the uniqueness of the initial value
problem~(\ref{1.8}) and~(\ref{1.9}) with $l=1$. Therefore,
$\Phi(r;k)$ is entire function of the same type and order as
$y(r;k)$.

\par
To find the estimates on the solution $y(r;k)$. We use the
Liouville transform.
\begin{eqnarray} \label{1.26}
&z(\xi;k):=[n(r)]^{\frac{1}{4}}y(r;k),\mbox{ where
}\xi:=\int_0^r[n(\rho)]^{\frac{1}{2}}d\rho.
\end{eqnarray}
Moreover, if we set
\begin{equation}
B:=\int_0^1[n(\rho)]^{\frac{1}{2}}d\rho,
\end{equation}
then
\begin{eqnarray}
\left\{%
\begin{array}{ll}
    z''+[k^2-p(\xi)]z=0; \\
    z(0)=0;\,(-k)z(B;k)j_1'(k)+z'(B;k)j_1(k)-z(B;k)j_1(k)=0, \\
\end{array}%
\right.
\end{eqnarray}
where
\begin{equation}\label{1.29}
p(\xi):=\frac{n''(r)}{4[n(r)]^2}-\frac{5}{16}\frac{[n'(r)]^2}{[n(r)]^3}+\frac{2}{r^2n(r)}.
\end{equation}
We rephrase the system above again.
\begin{eqnarray}\label{1.30}
\left\{%
\begin{array}{ll}
    z''+[k^2-q(\xi)-\frac{2}{\xi^2}]z=0; \\
    z(0)=0;\,D(k)=(-k)z(B;k)j_1'(k)+z'(B;k)j_1(k)-z(B;k)j_1(k)=0, \\
\end{array}%
\right.
\end{eqnarray}
where
\begin{equation}
q(\xi):=\frac{n''(r)}{4[n(r)]^2}-\frac{5}{16}\frac{[n'(r)]^2}{[n(r)]^3}+\frac{2}{r^2n(r)}-\frac{2}{\xi^2}.
\end{equation}
The fundamental estimates of its solution is found in Somasundaram
\cite[Lemma 5.5]{Somasundaram} which is based on the methods in
\cite{Po}. The (5.25) in \cite{Somasundaram} needs to be
considered subject to its (5.27) on page 45. In particular, we
need the following estimates to~(\ref{1.30}). For $|k|>1$,
\begin{eqnarray}\label{1.32}
z(\xi;k)&=&\frac{3\sin(k\xi)}{k^3\xi}-\frac{3\cos(k\xi)}{k^2}+O(\frac{\exp\{8\|q\|\sqrt{\xi}\}\exp\{|\Im k|\xi\}}{|k|^3});\\
z_\xi(\xi;k)&=&\frac{3\sin(k\xi)}{k^2}(\frac{-1}{k\xi^2}+k)+\frac{3\cos(k\xi)}{k^2\xi}+O(\frac{\exp\{8\|q\|\sqrt{\xi}\}\exp\{|\Im
k|\xi\}}{|k|^2}). \label{1.33}
\end{eqnarray}
They are bounded over $0i+\mathbb{R}$. We will apply Cartwright's
theory to such entire functions.

\par
May we ask that if the set of the interior transmission
eigenvalues of~(\ref{1.2}), in particular, the set of interior
transmission eigenvalues of the acoustic system~(\ref{1.4}) or
zeros of $D(k)$, can uniquely determine the refraction index
$n(r)$?

Following the local uniqueness results in \cite{Mc,Mc3}, we state
the uniqueness result in this paper.
\begin{theorem}\label{111}
Let the functional determinant $D^i(z)$, $i=1,2$, be defined by
refraction index $n^i(r)$ as in~(\ref{1.12}) and~(\ref{113}). If
the zeros of $D^i(z)$ inside any of the angular wedges
\begin{eqnarray}
&&\Sigma_1:=\{k\in\mathbb{C}|\,-\epsilon\leq\arg{k}\leq\epsilon\},\\
&&\Sigma_2:=\{k\in\mathbb{C}|\,\pi-\epsilon\leq\arg{k}\leq\pi+\epsilon\},\forall\epsilon>0,
\end{eqnarray}
coincide, then $n^1(r)\equiv n^2(r)$.
\end{theorem}

\section{Counting the Zeros: Cartwright's Theorem}
From~(\ref{1.12}), we compute the $D(k)$ as follows.
\begin{equation}\label{2.1}
D(k)=(-k)z(B;k)j_1'(k)+z'(B;k)j_1(k)-z(B;k)j_1(k),
\end{equation}
where
\begin{equation}
j_1(z)=\frac{\sin z}{z^2}-\frac{\cos z}{z},
\end{equation}
which is an entire function of order $1$. Moreover, using the
asymptotics~(\ref{1.32}) and~(\ref{1.33}), we compute the
following asymptotics.
\begin{eqnarray}\nonumber
D(k)&=&(-k)z(B;k)[\frac{2\cos(k)}{k^2}+(\frac{1}{k}-\frac{2}{k^3})\sin{k}]
+z'(B;k)[\frac{\sin{k}}{k^2}-\frac{\cos{k}}{k}]-z(B;k)[\frac{\sin{k}}{k^2}-\frac{\cos{k}}{k}]\\
&=&(-k)z(B;k)[\frac{\cos{k}}{k^2}+(\frac{1}{k}-\frac{1}{k^3})\sin{k}]
+z'(B;k)[\frac{\sin{k}}{k^2}-\frac{\cos{k}}{k}]\\
&=& \frac{3\cos(k B)}{k}[\frac{\cos
k}{k^2}+(\frac{1}{k}-\frac{1}{k^3})\sin
k][1+O(\frac{1}{k})]\nonumber\\&&+\frac{3\sin(k B)}{k}[\frac{\sin
k}{k^2}-\frac{\cos k}{k}][1+O(\frac{1}{k})],\,\forall k\notin
0i+\mathbb{R}, \label{2.4}
\end{eqnarray}
where we have used the fact that $\tan z$ and $\cot z$ are bounded
outside $0i+\mathbb{R}$. \par Moreover, we define
\begin{eqnarray} \nonumber
\mathcal{D}(k)&:=&\frac{k^4D(k)}{3}\\
&=&\cos(Bk)\cos k\{[k+(k^2-1)\tan k][1+O(\frac{1}{k})]\\&&+\tan(B
k)[k\tan k-k^2](1+O(\frac{1}{k}))\},\,\forall k\notin
0i+\mathbb{R}. \label{2.5}
\end{eqnarray}

\par
For such a representation form of an entire function, we consider
one type of the theorems concerning the distribution of the zeros
of certain class of entire functions. We apply the Cartwright's
theory. We refer the Cartwright's theory to the Levin's book
\cite{Levin,Levin2} and \cite{Cartwright, Cartwright2}. Let us
review the following verbatim.
\begin{definition}
Let $f(z)$ be an entire function. Let
$M_f(r):=\max_{|z|=r}|f(z)|$. An entire function of $f(z)$ is said
to be a function of finite order if there exists a positive
constant $k$ such that the inequality
\begin{equation}
M_f(r)<e^{r^k}
\end{equation}
is valid for all sufficiently large values of $r$. The greatest
lower bound of such numbers $k$ is called the order of the entire
function $f(z)$. By the type $\sigma$ of an entire function $f(z)$
of order $\rho$, we mean the greatest lower bound of positive
number $A$ for which asymptotically we have
\begin{equation}
M_f(r)<e^{Ar^\rho}.
\end{equation}
That is
\begin{equation}
\sigma:=\limsup_{r\rightarrow\infty}\frac{\ln M_f(r)}{r^\rho}.
\end{equation}  If $0<\sigma<\infty$, then we say
$f(z)$ is of normal type or mean type.
\end{definition}
We also see that
\begin{equation}
e^{(\sigma-\epsilon)r^\rho}\underset{{\rm
n}}{<}M_f(r)\underset{{\rm as}}{<}e^{(\sigma+\epsilon)r^\rho},
\end{equation}
where the first inequality holds for some sequence going to
infinity and the second one holds asymptotically.

\begin{definition}
If an entire function $f(z)$ is of order one and of normal type,
then we say it is an entire function of exponential type $\sigma$.
\end{definition}

\begin{definition}
Let $\rho\in\mathbb{R}$ and
$\rho(r):\mathbb{R}^+\rightarrow\mathbb{R}^+$. We say $\rho(r)$ is
a proximate order to $\rho$ if
\begin{equation}
\lim_{r\rightarrow\infty}\rho(r)=\rho\geq0;\,\lim_{r\rightarrow\infty}r\rho'(r)\ln
r=0.
\end{equation}
\end{definition}
\begin{definition}
Let $f(z)$ be an integral function of finite order in the angle
$[\theta_1,\theta_2]$. We call the following quantity as the
generalized indicator of the function $f(z)$.
\begin{equation} \label{2.45}
h_f(\theta):=\limsup_{r\rightarrow\infty}\frac{\ln|f(re^{i\theta})|}{r^{\rho(r)}},\,\theta_1\leq\theta\leq\theta_2,
\end{equation}
where $\rho(r)$ is some proximate order.
\end{definition}
The order and the type of an integral function in an angle can be
defined similarly. The connection between the indicator
$h_f(\theta)$ and its type $\sigma$ is specified by the following
theorem.
\begin{lemma}[Levin \cite{Levin},\,p.72]\label{2.5}
The maximum value of the indicator $h_f(\theta)$ of the function
$f(z)$ on the interval $\alpha\leq\theta\leq\beta$ is equal to the
type $\sigma_f$ of this function inside the angle $\alpha\leq\arg
z\leq\beta$.
\end{lemma}
\begin{lemma}\label{26}
Let $a$, $b$ be real constants.
\begin{eqnarray}
&&h_{\sin\{az+b\}}(\theta)=|a\sin\theta|; \label{2.39}
\end{eqnarray}
if $p(z)$ is a polynomial with bounded holomorphic coefficients,
then
\begin{equation}
h_{p(z)}(\theta)=0.
\end{equation}
\end{lemma}
\begin{proof}
We apply definition~(\ref{2.45}), we prove the lemma. $\Box$
\end{proof}
We mention two more inequalities for indicator functions.
\begin{lemma}\label{27}
Let $f$, $g$ be two entire functions. Then, the following two
inequalities hold.
\begin{eqnarray}\label{2.46}
&&h_{fg}(\theta)=h_{f}(\theta)+h_g(\theta),\mbox{ if one limit exists};\\
&&h_{f+g}(\theta)\leq\max_\theta\{h_f(\theta),h_g(\theta)\},
\label{2.47}
\end{eqnarray}
where if the indicator of the two summands are not equal at some
$\theta_0$, then the equality holds in~(\ref{2.47}).
\end{lemma}
\begin{proof}
 We can find
these in \cite[p.51]{Levin}. $\Box$
\end{proof}
\begin{definition}

The following quantity is called the width of the indicator
diagram of entire function $f$:
\begin{equation}\label{d}
\tilde{d}=h_f(\frac{\pi}{2})+h_f(-\frac{\pi}{2}).
\end{equation}
\end{definition}
\begin{definition}
Let $f(z)$ be an entire function of order $\rho(r)$. We use
$N(f,[\alpha,\beta],r)$ to denote the number of the zeros of
$f(x)$ inside the angle $[\alpha,\beta]$ and $|z|\leq r$; we
define the density function
\begin{equation}
\Delta_f(\alpha,\beta):=\limsup_{r\rightarrow\infty}\frac{N(f,[\alpha,\beta],r)}{r^{\rho(r)}},
\end{equation}
and
\begin{equation}
\Delta(\beta):=\Delta(\alpha_0,\beta),
\end{equation}
with fixed $\alpha_0\notin E$ with $E$ as an at most countable
set.
\end{definition}
The distribution on the zeros of an entire function is described
precisely by the following Cartwright's theorem
\cite{Cartwright,Cartwright2,Levin,Levin2}. The following
statements are from Levin \cite[ch.5, sec.4]{Levin}.
\begin{theorem}[Cartwright] \label{2.9}
If an entire function of exponential type satisfies one of the
following conditions:
\begin{equation}
\mbox{ the integral
}\int_0^\infty\frac{\ln|f(x)f(-x)|}{1+x^2}dx\mbox{ exists,\,and
}h_f(0)=h_f(\pi)=0,
\end{equation}
\begin{equation}
\mbox{ the integral
}\int_{-\infty}^\infty\frac{\ln|f(x)|}{1+x^2}dx<\infty.
\end{equation}
\begin{equation}
\mbox{ the integral
}\int_{-\infty}^\infty\frac{\ln^+|f(x)|}{1+x^2}dx\mbox{ exists}.
\end{equation}
\begin{equation}\label{2.17}
|f(x)|\mbox{ is bounded on the real axis}.
\end{equation}
\begin{equation}
|f(x)|\in\mathcal{L}^p(-\infty,\infty),
\end{equation}
then
\begin{enumerate}
    \item $f(z)$ is of class A and is of completely regular growth
    and its indicator diagram is an interval on the imaginary
    axis. In particular, for some constant $\kappa$, we have
    \begin{equation}\label{2.49}
    h_f(\theta)=\kappa\sin\theta;
    \end{equation}
\item all of the zeros of the function $f(z)$, except possibly
those of a set of zero density, lie inside arbitrarily small
angles $|\arg z|<\epsilon$ and $|\arg z-\pi|<\epsilon$, where the
density
\begin{equation} \label{249}
\Delta_i=\lim_{r\rightarrow\infty}\frac{n_i(r)}{r},\,i=1,2,
\end{equation}
of the set of zeros within each of these angles is equal to
$\frac{\tilde{d}}{2\pi}$, where $\tilde{d}$ is the width of the
indicator diagram in~(\ref{d}). Moreover, $n_i(r)$, $i=1,2$, here
is understood as the number of the zeros that fall in the wedge
$|\arg z|<\epsilon$ and $|\arg z-\pi|<\epsilon$ respectively.
Furthermore, the limit $\delta=\lim_{r\rightarrow\infty}\delta(r)$
exists, where
\begin{equation}
\delta(r):=\sum_{\{|a_k|<r\}}\frac{1}{a_k};
\end{equation}
\item moreover,
\begin{equation}\label{2.64}
\Delta_f(\epsilon,\pi-\epsilon)=\Delta_f(\pi+\epsilon,-\epsilon)=0,
\end{equation}
\item the function $f(z)$ can be represented in the form
\begin{equation}
f(z)=cz^me^{iC
z}\lim_{r\rightarrow\infty}\prod_{\{|a_k|<r\}}(1-\frac{z}{a_k}),
\end{equation}
where $c,m,B$ are constants and $C$ is real.
\end{enumerate}
\end{theorem}
Therefore, we apply this theorem to make the following conclusion.
\begin{proposition}
The indicator function of $\mathcal{D}(z)$ and $D(z)$ are equal
and
\begin{equation}\label{2.48}
h_{\mathcal{D}}(\theta)=(1+B)|\sin\theta|.
\end{equation}
\end{proposition}
\begin{proof}
We examine~(\ref{2.5}). When $\theta\neq0,\pi$, we see that $\tan
z$ and $\cot z$ are bounded functions. Hence, we can use the
product to sum formula to see that
\begin{eqnarray}\nonumber
\mathcal{D}(k)&=&\{\frac{1}{2}\cos[(B-1)k]+\frac{1}{2}\cos[(B+1)k]\}\\
&&\times\{[k+(k^2-1)\tan k][1+O(\frac{1}{k})]+\tan(B k)[k\tan
k-k^2](1+O(\frac{1}{k}))\}.
\end{eqnarray}
Hence, using Lemma \ref{26} and Lemma \ref{27},
\begin{eqnarray}
h_{\mathcal{D}}(\theta)=\max\{|1-B||\sin\theta|,\,(1+B)|\sin\theta|\}=(1+B)|\sin\theta|,\,\theta\neq0,\pi.
\end{eqnarray}
Given $\mathcal{D}(z)$ is entire, $h(\theta)$ is a continuous
function of $\theta$. We refer this to Levin \cite[p.54]{Levin}.
Hence, the statement is proved for $\mathcal{D}(z)$ for all
$\theta\in[0,2\pi]$. Surely, $D(z)$ and $\mathcal{D}(z)$ have the
same indicator function. $\Box$
\end{proof}
Moreover, we prove the following density theorem.
\begin{theorem}\label{2.12}
The length of the indicator diagram of $D(z)$ is $2(1+B)$. The
density in each of the two small angles along real axis is
\begin{equation}
\Delta_D(-\epsilon,\epsilon)=\Delta_D(\pi-\epsilon,\pi+\epsilon)=(1+B)/\pi.
\end{equation}
\end{theorem}
\begin{proof}
This follows from Cartwright's theorem as~(\ref{249}) and
definition~(\ref{d}). $\Box$
\end{proof}

\section{Proof of Theorem \ref{111}}
Let $D^i(z)$ be the functional determinant corresponding to
refraction index $n^i(r)$, $i=1,2$. If the zeros of $D^i$ in
either wedge coincide, then Theorem \ref{2.12} tells us that
\begin{equation}
B^1=B^2,
\end{equation}
where $B^i:=\int_0^1\sqrt{n^i(\rho)}d\rho$. Let us consider
\begin{equation}  \label{3.1}
F(k):=y^1(1;k)-y^2(1;k).
\end{equation}
Let $k_j$ be a common zero of $D^1(k)$ and $D^2(k)$, then, using
the boundary condition in the third equation in
system~(\ref{1.13}),
\begin{equation}
F(k_j)=y^1(1;k_j)-y^2(1;k_j)=0.
\end{equation}
Moreover, we use~(\ref{1.32}) and~(\ref{2.39}) to obtain
\begin{equation}
h_F(\theta)\leq B^1|\sin\theta|.
\end{equation}
Therefore, its indicator diagram is
\begin{equation}\label{38}
h_{F}(-\frac{\pi}{2})+h_{F}(\frac{\pi}{2})\leq 2B^1.
\end{equation}
To consider an uniqueness problem in entire function theory, we
apply a generalized Carlson's theorem from Levin
\cite[p.190]{Levin}.
\begin{theorem}\label{C}
Let $F(k)$ be holomorphic and at most of normal type with respect
to the proximate order $\rho(r)$ in the angle $\alpha\leq\arg
k\leq\alpha+\pi/\rho$ and vanish on a set $N:=\{a_k\}$ in this
angle, with angular density $\Delta_N(\psi)$. Let
$$
H_N(\theta):=\pi\int_{\alpha}^{\alpha+\pi/\rho}\sin|\psi-\theta|d\Delta_N(\psi),
$$
when $\rho$ is integral. Then, if $F(k)$ is not identically zero,
\begin{equation}
h_F(\alpha)+h_F(\alpha+\pi/\rho)\geq
H_N(\alpha)+H_N(\alpha+\pi/\rho).
\end{equation}
\end{theorem}
In this paper, we consider $\rho\equiv1,\,\alpha=-\frac{\pi}{2}.$
We set the collection of interior transmission eigenvalues as
\begin{equation}
N:=\{k_1,k_2,\cdots\}.
\end{equation}
\par
From Theorem \ref{2.12} and~(\ref{2.64}),
\begin{eqnarray}
\Delta_{D^i}(-\epsilon,\epsilon)=\Delta_{D^i}(\pi-\epsilon,\pi+\epsilon)=(1+B^1)/\pi;\\
\Delta_{D^i}(\epsilon,\pi-\epsilon)=\Delta_{D^i}(\pi+\epsilon,-\epsilon)=0.
\end{eqnarray}
Therefore,
\begin{equation}
H_N(\theta)=\pi\int_{-\frac{\pi}{2}}^{\frac{\pi}{2}}\sin|\psi-\theta|d\Delta_K(\psi)=(1+B^1)|\sin\theta|,
\,\theta\in[-\frac{\pi}{2},\frac{\pi}{2}].
\end{equation}
We refer to \cite[p.91;\,Ch 2. Sec\,3.]{Levin} for a complete
introduction on this set indicator function. Hence,
\begin{equation}\label{313}
H_N(-\frac{\pi}{2})+H_N(\frac{\pi}{2})=2(1+B^1).
\end{equation}
Now~(\ref{38}),~(\ref{313}) and Theorem \ref{C} imply that
\begin{equation}\label{3.12}
y^1(1;k)\equiv y^2(1;k).
\end{equation}
Similar argument can show that
\begin{equation}\label{3.13}
(y^1)'(1;k)\equiv (y^2)'(1;k).
\end{equation}

\par
We seek to identify $n(r)$ by inverse Sturm-Liouville theorem as
in \cite{Aktosun}, \cite{Mc}, \cite{Mc3,Mc2}. In particular, we
apply the methods reviewed in \cite{Mc3}.
\begin{theorem}[McLaughlin]\label{3.2}
We consider the following Sturm-Liouville problem
\begin{eqnarray}
&&\label{3.8}z''+(k^2-q)z=0,\,0<x<1;\\
&&\label{3.9}z(0)=z(1)=0,
\end{eqnarray}
where $q\in L^2(0,1)$. For another boundary condition,
\begin{equation}
\label{3.10}z(0)=z'(1)=0.
\end{equation}
Suppose $q_1,\,q_2\in L^2(0,1)$ and,
$\lambda_n(q_1)=\lambda_n(q_2)$, the eigenvalues to~(\ref{3.8})
and~(\ref{3.9}), $\mu_n(q_1)=\mu_n(q_2)$, the eigenvalues
to~(\ref{3.8}) and~(\ref{3.10}), $\forall n\in\mathbb{N}$. Then,
$q_1\equiv q_2$, a.e.
\end{theorem}

Under the Liouville transform~(\ref{1.26}), the zeros of $z(B;k)$
exactly corresponds to the eigenvalues of the Sturm-Liouville
problem
\begin{eqnarray} \label{3.15}
\left\{%
\begin{array}{ll}
    z''+[k^2-p(\xi)]z=0,\,0<\xi<B; \\
    z(0)=0;\,z(B)=0. \\
\end{array}%
\right.
\end{eqnarray}
Similarly, the zeros of $z'(B;k)$ exactly corresponds to the
eigenvalues of the Sturm-Liouville problem
\begin{eqnarray}\label{3.16}
\left\{%
\begin{array}{ll}
    z''+[k^2-p(\xi)]z=0,\,0<\xi<B;\\
    z(0)=0;\,z'(B)=0. \\
\end{array}%
\right.
\end{eqnarray}
Now $z^i(B;k)$ and $(z^i)'(B;k)$  corresponding to refraction
index $n^i(r)$ have common zeros by~(\ref{3.12}) and~(\ref{3.13}).
Hence, the Sturm-Liouville problems
\begin{eqnarray} \label{3.15}
\left\{%
\begin{array}{ll}
    (z^i)''+[k^2-p^i(\xi)]z^i=0,\,0<\xi<B; \\
    z^i(0)=0;\,z^i(B)=0, \\
\end{array}%
\right.
\end{eqnarray}
have the same eigenvalues for both $i=1,2$. Similarly,
\begin{eqnarray}\label{3.16}
\left\{%
\begin{array}{ll}
    (z^i)''+[k^2-p^i(\xi)]z^i=0,\,0<\xi<B;\\
    z^i(0)=0;\,(z^i)'(B)=0. \\
\end{array}%
\right.
\end{eqnarray}
have the same eigenvalues for both $i=1,2$. Given $p^i\in
\mathcal{L}^2$, $i=1,2$, Theorem \ref{3.2} says these imply
$p^1\equiv p^2$ almost everywhere. However,~(\ref{1.1})
and~(\ref{1.29}) implies $p^1\equiv p^2$. The solution is unique,
so we have
\begin{equation}
z^1(\xi;k)\equiv z^2(\xi;k),\,\forall\xi,k.
\end{equation}
This says that
\begin{equation}
[n^1(r)]^{\frac{1}{4}}y^1(r;k)\equiv
[n^2(r)]^{\frac{1}{4}}y^2(r;k),\,\forall r,k.
\end{equation}
Since $n^i$ never vanishes, the solutions $y^1(r;k),\,y^2(r;k)$
have common zero set in $\mathbb{C}$ of its maximal density as
described by Cartwright's theory for any fixed $r$.
From~(\ref{1.32}) and~(\ref{2.39}), their common density is
$\frac{\int_0^r\sqrt{n^i(\rho)}d\rho}{\pi}$. Hence,
\begin{equation}
\int_0^r[n^1(\rho)]^{\frac{1}{2}}d\rho=\int_0^r[n^2(\rho)]^{\frac{1}{2}}d\rho,\,\forall
r.
\end{equation}
This implies that $n^1\equiv n^2$. $\Box$


\begin{thebibliography}{widest-label}
\bibitem{Aktosun}T. Aktosun, D. Gintides and V.G. Papanicolaou,
The uniqueness in the inverse problem for transmission eigenvalues
for the spherically symmetric variable-speed wave equation,
Inverse Problems, v.27, 115004(2011).
\bibitem{Cartwright}M.L. Cartwright, On the directions of Borel of functions which are regular and of finite order in an angle,
 Proc. London Math. Soc. ser.2 vol.38, 503-541(1933).
\bibitem{Cartwright2}M.L. Cartwright, Integral functions,
Cambridge University Press, 1956.
\bibitem{Colton}D. Colton and
 R. Kress, Inverse acoustic and electromagnetic scattering theory,
2nd ed. Applied mathemtical science, v.93, Springer-Verlag, 1998.
\bibitem{Levin}B. Ja. Levin, Distribution of zeros of entire
functions, revised edition, Translations of mathematical
mongraphs, American mathemtical society, 1972.

\bibitem{Levin2}B. Ja. Levin, Lectures on entire functions,
Translation of mathematical monographs, V.150, AMS, 1996.

\bibitem{Mc}J.R. McLaughlin and P.L. Polyakov, On the uniqueness
of a spherically symmetric speed of sound from transmission
eigenvalues, Jour. Differentical Equations, 107, 351-382(1994).
\bibitem{Mc3}J.R. McLaughlin, P.E. Sacks, M. Somasundaram, Inverse scattering in acoustic media using interior transmission
eigenvalues, Inverse problems in wave propagation (Minneapolis,
MN, 1995), 357-374, IMA Vol. Math. Appl., 90, Springer, New York,
1997.
\bibitem{Mc2}J.R. McLaughlin, Inverse spectral theory using nodal
points as data-a uniqueness result, Jour. Differentical Equations,
73, 354-362(1988).
\bibitem{Somasundaram}M. Somasundaram, Recovery of the refractive
index from transmission eigenvalues, PhD thesis, Rensselaer
Polytechinic Institute, 1995.
\bibitem{Po}J. P\"{o}schel and E. Trubowitz, Inverse spectral theory,
Academic Press, Orlando, 1987.
\end{thebibliography}
\end{document}